\documentclass[11pt,reqno]{amsart}
\usepackage{amsmath}
\usepackage{amsthm}
\usepackage{amssymb}
\usepackage{enumerate}
\usepackage{amscd}
\usepackage{color}
\usepackage{pb-diagram}
\usepackage{graphicx}
\usepackage{hyperref}
\usepackage{hypcap}

\usepackage[tmargin=1in,bmargin=1in,lmargin=1in,rmargin=1in]{geometry}
\theoremstyle{plain}
\newtheorem{thm}{Theorem}

\newtheorem{lemma}{Lemma}[section]
\newtheorem{prop}[lemma]{Proposition}
\newtheorem{proposition}[lemma]{Proposition}

\newtheorem{coro}[lemma]{Corollary}
\newtheorem{corollary}[lemma]{Corollary}

\newtheorem{theorem}[lemma]{Theorem}

\newtheorem*{prop*}{Proposition}
\newtheorem*{thm-rigidityE}{Theorem~\ref{th:rigidityE}}

\theoremstyle{definition}
\newtheorem{defn}[lemma]{Definition}
\newtheorem{remark}[lemma]{Remark}
\newtheorem{ex}[lemma]{Example}

\newtheorem*{ack}{Acknowledgement}

\theoremstyle{remark}

\newtheorem*{pf}{Proof}
\numberwithin{equation}{section}
\newenvironment{enumeratei}{\begin{enumerate}[\upshape (i)]}{\end{enumerate}}

\newcommand{\Ric}{\textup{Ric}}

\newcommand{\inte}{\textup{int}}
\newcommand{\p}{\partial}

\begin{document}
\title[Geometric inequalities and rigidity theorems]{Geometric inequalities and rigidity theorems on equatorial spheres}
\author{Lan-Hsuan Huang}
\address{Department of Mathematics\\
 University of Connecticut\\
 Storrs, CT 06269, USA\\}
\email{lan-hsuan.huang@uconn.edu}

\author{Damin Wu}
\address{Department of Mathematics\\
 University of Connecticut\\
 Storrs, CT 06269, USA\\}

\email{damin.wu@uconn.edu}

\thanks{The authors were partially supported by National Science Foundation through DMS-1308837. The first named author was also partially supported by DMS-1452477.}
\begin{abstract}
We prove rigidity for hypersurfaces with boundary in the unit sphere $\mathbb{S}^{n+1}$ with scalar curvature $R\ge n(n-1)$. Under appropriate boundary conditions, the hypersurfaces are shown to be part of the equatorial spheres. The lower bound $n(n-1)$ is critical in the sense that the hypersurface may contain geodesic points and some natural differential operators  are fully degenerate at geodesic points.  We overcome the difficulty by  studying the geometry of level sets of a height function, via new geometric inequalities. Some rigidity results of hyperplanes and generalized cylinders are also obtained for hypersurfaces with boundary and with nonnegative scalar curvature in Euclidean space.
\end{abstract}
\maketitle

\section{Introduction}

Hypersurfaces in $\mathbb{S}^{n+1}$ of either \emph{constant scalar curvature} or \emph{constant mean curvature} have been extensively studied in the literature.  The pioneering work of S.~Y.~Cheng and S.~T.~Yau~\cite{Cheng-Yau:1977} classifies hypersurfaces that have constant scalar curvature and nonnegative sectional curvature. They introduced a self-adjoint operator which has been used by other people to study hypersurfaces of constant scalar curvature, under various conditions (see, for example, \cite{Alencar-doCarmo-Colares:1993, Li:1996}). 

In this paper, we consider hypersurfaces $M$ with nonempty boundary in $\mathbb{S}^{n+1}$ for $n\ge 2$ with  scalar curvature $R\ge n(n-1)$. We prove rigidity of $M$ under a suitable boundary condition. The lower bound $n(n-1)$ is a critical value because $R\ge n(n-1)$ implies that $H^2 \ge |A|^2$ by Gauss equation, so it is possible for the mean curvature to change signs at the geodesic points.  From an analytic aspect, several geometric operators, such as Cheng--Yau's operator, the linearized scalar curvature operator, and the scalar curvature flow, are no longer globally elliptic or parabolic and are fully degenerate at the geodesic points. Thus the theory of maximum principle is not applicable to those operators. (We remark that if one assumes $M$ is contained in the hemisphere and has constant scalar curvature, then the desired strict ellipticity automatically holds; hence, $M$ is a sphere by applying the Alexandrov reflection principle (see N.~Korevaar~\cite{Korevaar:1988}).)
In this article, we develop a different approach using new geometric inequalities and the level sets of hypersurfaces, motivated by our recent work for \emph{complete} hypersurfaces in Euclidean space~\cite{Huang-Wu:2013}.

Another motivation of this paper comes from Min--Oo's problem. Let $M$ be an $n$-dimensional compact Riemannian manifold of scalar curvature $R \ge n(n-1)$ with boundary $\partial M$. Suppose $\partial M$ is isometric to the unit sphere $\mathbb{S}^{n-1}$ and is totally geodesic in $M$. The problem asks whether $M$ is isometric to the hemisphere $\mathbb{S}^n_+$.  Recently the counter-examples for the general statement are provided by S.~Brendle, F.~Marques, and A.~Neves~\cite{Brendle-Marques-Neves:2011}. Nevertheless, there have been many interesting positive results in various settings. F.~Hang and X.~Wang \cite{Hang-Wang:2006, Hang-Wang:2009} proved the conjecture under the condition that, either $g$ is conformal to the standard sphere metric, or the Ricci curvature satisfies $\mbox{Ric} \ge (n-1)g$. By assuming positive Ricci curvature on $M$ and an isoperimetric condition on the boundary, M.~Eichmair  proved the conjecture in three dimensions~\cite{Eichmair:2009}.  We refer to a survey by Brendle~\cite{Brendle:2012} and the references therein.

In an early work~\cite{Huang-Wu:2010}, we confirmed the conjecture if $M$ is a hypersurface with boundary in either Euclidean space or the hyperbolic space, by applying the strong maximum principle to the mean curvature operator. It is a natural continuation to study the case when $M$ is a hypersurface in $\mathbb{S}^{n+1}$. However, the method in \cite{Huang-Wu:2010} does not apply to the spherical case largely due to the failure of ellipticity as discussed above. We overcome the difficulty by a geometric argument and obtain the following result. Denote by $\mathbb{S}^{k}$ a $k$-dimensional unit sphere in the unit sphere $\mathbb{S}^{n+1}$ and by $\mathbb{S}^k_+$ the (closure) of the $k$-dimensional hemisphere in $\mathbb{S}^{n+1}$.  

\begin{thm} \label{th:minoo}
Let $n \ge 2$. Let $M$ be a connected, embedded, two-sided hypersurface in $\mathbb{S}^{n+1}$  with boundary $\p M$. Suppose $\textup{int}(M)$ is $C^{n+1}$ and $M$ is $C^1$  up to boundary. Suppose $M$ and $\partial M$ satisfy the following conditions:
\begin{enumerate}
\item $M$ satisfies $R\ge n(n-1)$; 
\item $\p M$ is a great $(n-1)$-sphere $\mathbb{S}^{n-1}$;
 \item  $M$  is tangent to a great $n$-sphere $\mathbb{S}^n$ at $\p M$. 
\end{enumerate}
 Then $M$ is  the hemisphere $\mathbb{S}^n_+$. 
\end{thm} 
Theorem~\ref{th:minoo} confirms Min--Oo's problem for  hypersurfaces in $\mathbb{S}^{n+1}$ under a slightly weaker boundary condition, namely, $\partial M$ is not assumed totally geodesic in $M$. Theorem~\ref{th:minoo} is implied by the following more general result.  

\begin{thm} \label{th:rigidityS}
Let $n \ge 2$.  Let $M$ be a connected, embedded, and two-sided hypersurface in $\mathbb{S}^{n+1}$  with boundary $\p M$. Suppose $\textup{int}(M)$ is $C^{n+1}$ and $M$ is $C^1$  up to boundary. Suppose $M$ and $\p M$ satisfy the following conditions:
\begin{enumerate}
\item $M$ satisfies $R\ge n(n-1)$; 
\item  \label{it:sphere} $\partial M$ is contained in $\mathbb{S}_+^n$;
\item \label{it:rigidityS2}$M$ is tangent to $\mathbb{S}_+^n$ at $\partial M$ from the region enclosed by $\partial M$.  
\end{enumerate}
 Then $M$ is a portion of the hemisphere $\mathbb{S}^n_+$. 
\end{thm} 

We remark that $\partial M$ is a $C^0$ submanifold, but we do not need any additional regularity of $\partial M$.  The interior regularity that $\textup{int}(M)\in C^{n+1}$ is needed to apply the Sard theorem in Lemma~\ref{le:geodesicS}. It is of interest whether the regularity assumption can be weakened to $\textup{int}(M)\in C^2$.

The key ingredient is a geometric inequality for the level sets of a height function in $M$, that holds for a large class of ambient spaces. 
\begin{thm}\label{th:HHRphi}
Let $(N,g)$ be an $n$-dimensional Riemannian manifold.  Suppose $M$ is a $C^2$ hypersurface in the product manifold $(N \times \mathbb{R}, g + dt^2)$ with a unit normal vector field $\nu$. Let  $\Sigma =  M \cap t^{-1}(\epsilon)$ be a regular level set in $M$, and let $\eta$ be a unit normal to $\Sigma \subset (N \times \{\epsilon\}, g)$.  Consider the conformal metric  $\phi^{-2} (g + dt^2)$ on $N\times \mathbb{R}$ for a positive function  $\phi \in C^1(N\times \mathbb{R})$. Denote by $\bar{A}$ the shape operator of $M$ in $(N\times \mathbb{R}, \phi^{-2}(g+dt^2))$ with respect to $\phi \nu$ and by $\bar{A}_{\Sigma}$ the shape operator  $\Sigma$ in $(N \times \{\epsilon\}, \phi^{-2}(\cdot, \epsilon) g)$ with respect to $\phi \eta$. Denote by $\bar{H}$, $\bar{H}_{\Sigma}$ the corresponding mean curvature scalars. Then the following inequality holds on $\Sigma$
  \begin{equation} \label{eq:HHRphi}
    \begin{split}
	&\bar{H} \left[ \langle \nu, \eta\rangle \bar{ H}_{\Sigma} +(n-1) \langle \nu, \p_{t}\rangle \phi_{t} \right ] \\
	&\ge \frac{1}{2}\left( \bar{H}^2 - |\bar{A}|^2\right) + \frac{n}{2(n-1)} \left[ \langle \nu, \eta\rangle \bar{ H}_{\Sigma} +(n-1) \langle \nu, \p_t\rangle \phi_{t} \right ]^2,
	  \end{split}
	\end{equation}
	where $\phi_t = \p \phi/\p t$, and $\langle \cdot, \cdot \rangle$ is taken with respect to $g + dt^2$. The equality in \eqref{eq:HHRphi} holds at $p\in \Sigma$ if and only if $M$ and $\Sigma$ satisfy the following conditions at $p\in \Sigma$:
\begin{enumeratei}
  \item $\Sigma$ is umbilic at $p$ in $(N\times\{\epsilon\}, \phi^{-2}(\cdot,\epsilon)g)$. Denote by $\kappa$ the principle curvature of $\Sigma$.
  \item $M \subset (N \times \mathbb{R}, \phi^{-2}g)$ has principal curvature $\langle \nu,\eta\rangle \kappa+\langle \nu, \p_t \rangle \phi_t$ with multiplicity at least $n-1$.
\end{enumeratei}
\end{thm}

As an independent result, we show rigidity for hypersurfaces with boundary in $\mathbb{R}^{n+1}$ with nonnegative scalar curvature. Let $P_+$ be  a half-hyperplane in $\mathbb{R}^{n+1}$, e.g.  $\{(x_1, \dots, x_n, 0) \in \mathbb{R}^{n+1}: x_1\ge 0 \}$.
\begin{thm} \label{th:rigidityE}
Let $n \ge 2$.  Let $M$ be a connected,  embedded, two-sided hypersurface in $\mathbb{R}^{n+1}$ with nonnegative scalar curvature and with nonempty boundary $\p M$ (the boundary need not be bounded). Suppose $\textup{int}(M)$ is $C^{n+1}$ and $M$ is $C^1$  up to boundary. Suppose that $\partial M \subset P_+$ and $M$ is tangent to a hyperplane at $\p M$ from the region enclosed by $\partial M$ in $P_+$. Then $M$ is the portion of either a hyperplane or a generalized cylinder.
\end{thm}

The paper is organized as follows. The geometric inequalities are derived in Section \ref{section:geometric-inequality}. The rigidity theorem for hypersurfaces in the sphere is proved in Section \ref{se:rigidityS} and the Euclidean case is proved in Section \ref{se:rigidityE}. In Section~\ref{se:Examples}, we demonstrate by two examples that the condition in Theorem~\ref{th:rigidityS} and Theorem~\ref{th:rigidityE} that $M$ is tangent at $\partial M$ from the region \emph{enclosed} by $\partial M$ is necessary.
\begin{ack}
In an earlier version of this paper, we proved a weaker version of Theorem \ref{th:minoo} using the mean curvature flow. We would like to thank Gerhard Huisken and Tom Ilmanen for helpful discussions along that direction. We also thank Pengzi Miao for his helpful comments and kind encouragement. The first author is grateful to Panagiota Daskalopoulos for discussions and to the Albert Einstein Institute for their hospitality and generous support. The second author would like to thank Jianguo Cao and Brian Smyth for the conversations. 
\end{ack}


\section{Geometric inequalities} \label{section:geometric-inequality}

\subsection{Product manifolds}\label{se:prod}
Consider the product manifold $(N\times \mathbb{R}, g + dt^2)$, where $(N, g)$ is a $n$-dimensional Riemannian manifold.  Let $M$ be a $C^{2}$ hypersurface in the product manifold $N\times \mathbb{R}$ endowed with the induced metric $g^M$, and let $\Sigma = M\cap t^{-1}(\epsilon)$ be a regular level set.  Let $\nu$ and $\eta$ be vector fields in the tangent space of $N\times \mathbb{R}$ along $\Sigma$ such that $\nu$ is a unit normal to $M$, and $\eta$ is a unit normal to $\Sigma$ and $\partial t$.  Denote by $A$ the shape operator of $M$ with respect to $\nu$, and by $A_\Sigma$ the shape operator of $\Sigma\subset N\times \{ \epsilon\}$ with respect to $\eta$. Let $H$ and $H_\Sigma$ be the corresponding mean curvature scalars. The mean curvature is defined as the trace of the shape operator; equivalently, the \emph{negative} divergence of the unit normal vector field. 

\begin{lemma} \label{le:sgm1A}
Let $(e_1, e_2, \dots, e_n)$ be an orthonormal frame in the neighborhood of a point in $ \Sigma$ in  $M$ such that $(e_2, \dots, e_n)|_\Sigma$  are tangent to $\Sigma$. For $i,j = 2,\dots, n$,
\[
      A^i_j = \langle \nu, \eta \rangle (A_{\Sigma})^i_j
\]
where the inner product $\langle \cdot, \cdot \rangle$ is with respect to $g + dt^2$.
\end{lemma}
\begin{proof}
Note that $\nu = \langle \nu, \eta\rangle \eta + \langle \nu, \partial_t\rangle \partial_t$.  For $i,j=2,\dots, n$, because $\langle \eta, e_j\rangle =0$, $\langle \partial_t, e_j\rangle=0$, and $\partial_t$ is parallel along $e_i$, we have
\[
	\langle \nabla_{e_i} \nu, e_j\rangle = \langle \nabla_{e_i} \left( \langle \nu, \eta\rangle \eta + \langle \nu, \partial_t\rangle \partial_t\right), e_j\rangle = \langle \nu, \eta \rangle  \langle \nabla_{e_i}\eta, e_j \rangle.
\]
\end{proof}

\begin{prop} [\cite{Huang-Wu:2013}]\label{pr:id}
  Let $A = (a_{ij})$ be a real $n \times n$ matrix with $n \ge 2$. Denote
\[
   \sigma_1(A) = \sum_{i=1}^n a_{ii}, \quad \sigma_1(A|1) = \sum_{i=2}^n a_{ii}, \quad \sigma_2(A) = \sum_{1 \le i < j \le n} (a_{ii} a_{jj} - a_{ij}a_{ji}). 
\]
  Then, we have
  \begin{align*}
    \sigma_1(A) \sigma_1(A|1) 
    & = \sigma_2(A) + \frac{n}{2(n-1)} [\sigma_1(A|1)]^2 + \sum_{1 \le i < j \le n} a_{ij} a_{ji} \\
    & \quad + \frac{1}{2(n-1)} \sum_{2 \le i < j \le n} (a_{ii} - a_{jj})^2.
  \end{align*}
  In particular, if $A$ is a symmetric matrix, then
  \[
     \sigma_1(A) \sigma_1(A|1) \ge \sigma_2(A) + \frac{n}{2(n-1)} [\sigma_1(A|1)]^2
  \]
with the equality holds if and only if $a_{22} = \cdots = a_{nn}$ and $a_{ij} = 0$ for all $1 \le i < j \le n$.
\end{prop}

\begin{theorem} \label{th:HHRprod}
Let $M$ be a $C^2$ hypersurface in the product manifold $(N\times \mathbb{R}, g+dt^2)$, and let $\Sigma = M\cap t^{-1}(\epsilon)$ be a regular level set. 
Denote by $\nu$ a unit normal to $M$ in $N\times \mathbb{R}$, and let $H $ be the corresponding mean curvature. Denote by $\eta$ a unit normal to $\Sigma$ in $N \times \{\epsilon\}$, and let $H_\Sigma$ be the corresponding mean curvature. Then we have the following inequality on  $\Sigma$ 
\begin{equation} \label{eq:HHRprod}
  \begin{split}
		\langle \nu, \eta \rangle H H_{\Sigma} &\ge \frac{1}{2}R(g^M) - \frac{1}{2}R(g) + \langle\nu, \eta \rangle^2 \Ric_{g} (\eta, \eta) +\frac{n}{2(n-1)} \langle \nu, \eta \rangle^2 H_{\Sigma}^2,
	\end{split}
\end{equation}
where the inner product $\langle \cdot, \cdot \rangle$ is with respect to the metric $g + dt^2$. The equality holds at  $p\in \Sigma$ if and only if the following holds at $p\in \Sigma$:
\begin{enumeratei}
  \item $\Sigma$ is umbilic in $N\times\{\epsilon\}$. Denote by $\kappa$ the principal curvature of $\Sigma$.
  \item $M \subset N \times \mathbb{R}$ has principal curvature $\langle \nu,\eta\rangle \kappa$ with multiplicity at least $n-1$.
\end{enumeratei}
\end{theorem}
\begin{proof}
By Lemma~\ref{le:sgm1A}, Proposition~\ref{pr:id}, and $2\sigma_2(A) = H^2 - |A|^2$,
\[
	\langle \nu, \eta\rangle  H H_\Sigma \ge \frac{1}{2}( H^2 - |A|^2) + \frac{n}{2(n-1)}\langle \nu, \eta \rangle^2 H_{\Sigma}^2.
\]
The desired inequality then follows by applying the Gauss equation to $M$ in $N\times \mathbb{R}$ 
\[
	  R(g+dt^2) = 2 \Ric_{g+dt^2}(\nu,\nu) + R(g^M) - H^2 + |A|^2.
\]
and the curvature formulas of a product metric
\[
	R(g + dt^2 ) = R(g), \qquad  \Ric_{g+dt^2}(\nu,\nu)= \Ric_{g}(\nu', \nu'),
\]
where $\nu' = \nu - \langle \nu, \partial_t\rangle \partial_t =\langle \nu, \eta\rangle \eta $.
\end{proof}

Applying Theorem~\ref{th:HHRprod} to Euclidean space $\mathbb{R}^{n+1}$, we recover the formula in \cite[Theorem 2.2]{Huang-Wu:2013}.
\begin{coro} \label{co:eqE}
Let $M$ be a $C^2$ hypersurface in the Euclidean space $\mathbb{R}^{n+1}$. Let $\Sigma = M \cap \{ x_{n+1} = \epsilon\} $ be a regular level set. If $\nu$ and $\eta$ are unit normal vectors to $M\subset \mathbb{R}^{n+1}$ and $\Sigma \subset \{ x_{n+1} = \epsilon\}$, respectively,  let $H$ and $H_\Sigma$ be the corresponding mean curvature scalars. Denote by $R$ the induced scalar curvature of $M$. Then 
\begin{align*}
	 \langle \nu, \eta\rangle H H_{\Sigma}  \ge \frac{1}{2} R  + \frac{n}{2(n-1)}  \langle \nu, \eta\rangle^2 H_{\Sigma}^2,
\end{align*}
where the equality holds at  $p\in \Sigma$ if and only if the following conditions hold at $p\in \Sigma$:
\begin{enumeratei}
  \item $\Sigma$ is umbilic in $\{ x_{n+1} = \epsilon\} $. We denote the principal curvature of $\Sigma$ by $\kappa$.
  \item $M$ has a principal curvature $ \langle \nu,\eta \rangle \kappa$  with multiplicity at least $n-1$.
\end{enumeratei}
\end{coro}

\subsection{Conformal metrics}\label{se:conf}
We now generalize the geometric inequality to the conformal product metrics. Theorem~\ref{th:HHRphi} follows by Lemma~\ref{lemma:geometric-conformal} and Proposition~\ref{pr:id} below.

Let us first recall a general formula for the shape operator under conformal transformation. Let $g, \bar{g}$ be two Riemannian metrics on a $(n+1)$-dimensional manifold $X$ that are related by 
\[
	\bar{g} = \phi^{-2} g.
\]	
Let $M\subset X$ be a two-sided hypersurface. If $\nu$ is a unit normal vector with respect to $g$, then $\bar{\nu}=\phi \nu$ is a unit normal with respect to $\bar{g}$. If  $\bar{A}$ and $A$ are the corresponding  shape $(1,1)$ tensors, then, with respect to a frame $\{e_1, \dots, e_n\}$ of $M$, 
\begin{equation} \label{eq:Hbar}
   \bar{A}^i_j = \phi A^i_j + \nu(\phi) \delta^i_j,
\end{equation}
and the corresponding mean curvature scalars $\bar{H}$ and $H$ are related by 
\begin{align} \label{eq:mean-curvature}
	\bar{H} = \phi H + n \nu(\phi).
\end{align}
 
 We now apply the above discussion to a hypersurface $M \subset N\times \mathbb{R}$ and a regular level set $\Sigma = M\cap t^{-1}(\epsilon)$. Consider the product metric $g+dt^2$ on $N\times \mathbb{R}$. Let $\nu$ and $\eta$ be the unit normals to $M\subset N\times \mathbb{R}$ and to $\Sigma\subset N\times \{ \epsilon\}$, respectively. Let $A$ and $A_\Sigma$ be the corresponding shape operator and let $H$ and $H_\Sigma$ be the corresponding mean curvature scalars. 
 
Let $N\times \mathbb{R}$ be endowed with the conformal metric $\bar{g} = \phi^{-2}(g+dt^2)$. Denote by $\bar{A}$ and $\bar{A}_{\Sigma}$ the shape operators of $M\subset N\times \mathbb{R}$ and $\Sigma\subset N \times \{\epsilon\}$ with respect to the conformal metric, respectively. Let $\bar{H}$, $\bar{H}_{\Sigma}$ be the corresponding mean curvature scalars.
 
\begin{lemma}\label{lemma:geometric-conformal}
Let $(e_1, e_2, \dots, e_n)$ be an orthonormal frame with respect to $g+dt^2$  in a neighborhood of $M$ that contains $\Sigma$  such that $(e_2, \dots, e_n)|_\Sigma$  are tangent to $\Sigma$. The following holds on $\Sigma$, for $i,j = 2,\dots, n$, 
\[
	\bar{A}^i_j =  \langle \nu, \eta \rangle (\bar{A}_\Sigma)^i_j + \langle \nu, \p_t \rangle (\partial_t \phi) \delta^i_j, 
\]
where $\langle\cdot, \cdot \rangle$ is with respect to $g+ dt^2$.
\end{lemma}
\begin{proof}
By \eqref{eq:Hbar} and Lemma~\ref{le:sgm1A}, for $i,j=2,\dots, n$,
\begin{align*}
	\bar{A}^i_j &= \phi A^i_j + \nu(\phi) \delta^i_j = \phi \langle \nu, \eta \rangle (A_\Sigma)^i_j + \nu(\phi) \delta^i_j\\
	&=  \langle \nu, \eta \rangle  (\bar{A}_\Sigma)^i_j -  \phi \langle \nu, \eta \rangle \eta (\phi)\delta^i_j + \nu(\phi) \delta^i_j.
\end{align*}
The identity follows by $\nu =\phi \langle \nu, \eta \rangle \eta + \langle \nu, \partial_t \rangle \partial_t$.
\end{proof}

 For the spherical metric $g_S= \phi^{-2} g_0$ on $\mathbb{R}^{n+1}$,  we have the following corollary, where 
\[
	\phi = \frac{1+\sum_{i=1}^{n+1} (x^i)^2}{2}.
\]	 
\begin{coro} \label{co:eqS}
Let $M$ be a $C^2$ hypersurface in $(\mathbb{R}^{n+1}, g_S)$. Let $\Sigma = M \cap \{ x^{n+1} = \epsilon\} $ be a regular level set. Let $\nu$ and $\eta$ be unit normal vectors to $M\subset (\mathbb{R}^{n+1},g_0)$ and $\Sigma \subset (\mathbb{R}^n \times \{ x^{n+1}=\epsilon\}, g_0|_{\{ x^{n+1}=\epsilon\}})$, respectively. Let $H$ and $H_\Sigma$ be the mean curvature scalars with respect to $\phi \nu $ and $\phi \eta$, respectively. Then 
\begin{align*}
	& H  \left[ g_0 (\nu, \eta) H_{\Sigma} +(n-1) g_0(\nu, \p_{n+1}) x^{n+1} \right ] \\
	 & \ge \frac{1}{2} [R - n(n-1)] + \frac{n}{2(n-1)} \left[ g_0 (\nu, \eta) H_{\Sigma} +(n-1) g_0(\nu, \p_{n+1}) x^{n+1} \right ]^2,
\end{align*}
where the equality holds at  $p\in \Sigma$ if and only if the following conditions hold at $p\in \Sigma$:
\begin{enumeratei}
  \item $\Sigma$ is umbilic in $(\mathbb{R}^n\times\{\epsilon\},g_S)$. We denote the principal curvature of $\Sigma$ by $\kappa$.
  \item $M$ has a principal curvature $g_0 (\nu,\eta) \kappa + g_0( \nu, \p_{n+1} ) \phi_{n+1}$ in $(\mathbb{R}^{n+1},g_S)$ with multiplicity at least $n-1$.
\end{enumeratei}
\end{coro}

Note that the geometric inequality for the hyperbolic space also appears in \cite{Dahl-Gicquaud-Sakovich:2013} after a preprint of this article was available on the arXiv.

\section{Rigidity in spheres} \label{se:rigidityS}

Consider  the the spherical metric $g_S = \phi^{-2} g_0$ on $\mathbb{R}^{n+1}$ where
\begin{align}\label{eq:spherical-metric}
	\phi =\frac{ 1 + \sum_{i=1}^{n+1} (x^i)^2 }{2}.
\end{align}
The stereographic projection from $\mathbb S^{n+1}$ with a point removed onto $(\mathbb{R}^{n+1}, g_S)$ gives an isometry. For $u\in C^2(\mathbb{R}^n)$, let $H(u)$ be the mean curvature of the graph of $x_{n+1} = u(x_1, \dots, x_n)$  in $(\mathbb{R}^{n+1}, g_S)$ with respect to the upward unit normal vector $ \phi\nu$, where $\nu = \frac{(-Du, 1)}{\sqrt{1+ |Du|^2}}$. By \eqref{eq:mean-curvature} and direct computation, $H(u)$ is a quasi-linear elliptic operator:
\begin{align*}
	H(u) &= \frac{1+|x|^2 + u^2}{2} \sum_{i,j=1}^n \left( \delta_{ij} - \frac{u_i u_j}{1+ |Du|^2} \right) \frac{u_{ij}}{\sqrt{1+ |Du|^2}} + \frac{n}{\sqrt{1+|Du|^2}}\bigg(u - \sum_{i=1}^n x_i u_i \bigg).
\end{align*}
We apply the maximum principle for the mean curvature $H(u)$ (see \cite[Appendix A]{Huang-Wu:2010}) and compare with the graph of  $h(x) = v \cdot x$ for some constant vector $v \in \mathbb{R}^n$. The graph of $h$ is a great $n$-sphere in $(\mathbb{R}^{n+1}, g_S)$ and in particular has zero mean curvature. For $a>0$, define $X_a = \{ (x_1, \dots, x_n)\in \mathbb{R}^n: 0\le x_1 < a\}$. 

\begin{lemma} \label{lemma:pos}
Let $W$ be an open subset in $X_a$ for some $a>0$, and let $\partial W$ be the boundary of $W$ in $X_a$.  Let $u \in C^2( W) \cap C^1 (\overline{W})$ satisfy that,  for a constant vector $v \in \mathbb{R}^n$,  $u(x)=v\cdot x$ and $Du=v$ on $ \p W$.  If $H(u) \ge 0 $ in $ W$, then $u(x) > v\cdot x$ somewhere in $W$, unless  $u\equiv v\cdot x$ in $W$. 
\end{lemma}
\begin{proof}
Denote by $h(x) = v \cdot x$ on $X_a$.  Suppose on the contrary that $u \le h$ on $W$ and $u$ is not identically equal to $h$. If $u=h$ at a point in $ W$, then the  graph of $h$ is tangent to the graph of $u$ from above, which contradicts the strong maximum principle since  $H(u) \ge H(h) = 0$.  Thus $u < h$ in $W$. Let $B\subset W$ be an open ball such that $\partial W\cap \partial B \neq \emptyset$. Because $u < h$ in $B$ and $u(q)=h(q)$ at $q\in \partial W\cap \partial B$, the Hopf boundary point lemma implies that $Du(q) \neq  Dh(q) $ (see, for example, \cite[p. 359]{Huang-Wu:2010}). However, it contradicts the boundary condition.
\end{proof}

In the following rigidity proposition, we use the geometric inequality from Corollary \ref{co:eqS} to obtain mean curvature comparison between the level sets of  $u$ and the level sets of some special functions. The central idea is that the assumption on the scalar curvature and mean curvature of the graph implies certain control on the convexity of its level sets. 

\begin{prop} \label{pr:levelset}
Let $W$ be an open subset in $X_a$ for some $a>0$.  Denote by $\partial W$  the boundary of $W$ in  $X_a$. Let  $u \in C^2(W ) \cap C^1 (\overline{W})$ be a bounded function that satisfies $u(x)=v \cdot x$ and $Du=v$ on $\p W$ for some $v \in \mathbb{R}^n$. Suppose that the graph of $u$ in $(\mathbb{R}^{n+1}, g_S)$ satisfies $R\ge n(n-1)$. If either $H(u)\ge 0$ or $H(u)\le 0$ everywhere in $W$, then $u(x)= v \cdot x$ in  $W$, and in particular $H(u) \equiv 0$. 
\end{prop}

\begin{proof}
Since the rotation about the origin  is a rigid motion in $(\mathbb{R}^{n+1}, g_S)$, we may assume that $v = 0$. We also assume $H(u) \ge 0$ in $W$; otherwise, replace $u$ by $-u$. Suppose on the contrary that $u$ is not identically zero. For $b>0$ fixed and for the parameter $\lambda$, define a family of superlinear functions $\psi_{\lambda}(x) = \lambda (x_1)^{1+b}$.  Because of  the boundary condition of $u$ and that $u$ is bounded, for sufficiently large $\lambda$ we have $\psi_{\lambda} > u$ in  $W$. For a fixed $ a_0 \in (0, a)$,  continuously decrease $\lambda$ until, for the first time, $\psi_{\lambda} (p) = u(p)$ at some $p = (p_1, \dots, p_n)\in  W\cap \overline{X_{a_0}}$. By Lemma \ref{lemma:pos}, we know $u(p)>0$ and hence $\lambda > 0$. Either the graphs of $u$ and $\psi_\lambda$ are tangent over $p$ (so $D u(p) = D \psi_\lambda (p) $) or $p \in \{ (x_1, \dots, x_n)\in \mathbb{R}^n: x_1 = a_0\}$. In the latter case, since $p$ is the first contact point, 
\[
\partial_{x_1} u (p)= \lim_{t\to 0^+} \frac{u(p) - u(p-t \partial_{x_1})}{t} \ge  \lim_{t\to 0^+} \frac{\psi_\lambda(p) - \psi_{\lambda}(p-t \partial_{x_1})}{t} = \partial_{x_1} \psi_\lambda (p)= (1+b)\lambda p_1^b.
\]
In both cases, $|Du|(p)\ge(1+b)\lambda p_1^b>0$. 

Let $\epsilon = u(p)$. Consider the level set $\Sigma = \{ (x_1,\dots, x_n) \in u^{-1}(\epsilon) \cap W \mbox{ and }  x_{n+1} = \epsilon\}$. Because $|Du|(p) > 0$, $\Sigma$ is  $C^2$ near $p$.  Let $  \eta=(Du,0) / |Du|$, so $\phi \eta$ is a unit normal to $\Sigma$ in  $(\{ x^{n+1} = \epsilon \}, g_S\big|_{\{x^{n+1} = \epsilon \}})$. Let $H_{\Sigma}$ be the corresponding mean curvature. 
Since $H(u)\ge 0$ and $R\ge n(n-1)$, Corollary \ref{co:eqS} implies that 
\[
	0\le g_0(\nu, \eta) H_{\Sigma} + (n-1)g_0 (\nu, \partial_{n+1}) x^{n+1} = - \frac{|Du|}{\sqrt{1+|Du|^2}} H_{\Sigma} + \frac{(n-1)u}{\sqrt{1+|Du|^2}}.
\]
Thus, at $p$,
\[
	H_{\Sigma} \le \frac{(n-1) u}{|Du|} \le (n-1) \frac{\psi_{\lambda}(p)}{(1+b)\lambda p_1^b} = (n-1) \frac{p_1}{1+b}.
\]
Let $\widetilde{\Sigma}:=\{ (x_1,\dots, x_n)\in \psi_\lambda^{-1}(\epsilon) \mbox{ and } x_{n+1} = \epsilon\}$ be the level set of $\psi_\lambda$, which is a hyperplane in $\{ x_{n+1}= \epsilon\}$. By \eqref{eq:mean-curvature}, the mean curvature $H_{\widetilde{\Sigma}}$  of $\widetilde{\Sigma}\subset (\{x^{n+1} = \epsilon\}, g_S\big|_{\{x^{n+1} = \epsilon\}})$ with respect to $\phi \partial_{x_1}$ is
\[
   H_{\widetilde{\Sigma}}(p) = (n-1)(\partial_{x_1} \phi ) (p) =(n-1) p_1.
\]
On the other hand, by comparison principle $H_{\Sigma} \ge H_{\widetilde{\Sigma}}$ at $p$ since $\widetilde{\Sigma} $ is tangent to $\Sigma$ at $p$ toward their common normal vector $\phi \partial_{x_1}$ at $p$. A contradiction. 
\end{proof}

We shall apply the above proposition to hypersurfaces with $R\ge n(n-1)$ in the unit sphere $\mathbb S^{n+1}$  that satisfy certain boundary conditions. Denote by $\mathbb{S}^k$ the $k$-dimensional great sphere in $\mathbb S^{n+1}$ and by $\mathbb{S}^k_+$ the (closure of) hemisphere. 
 
\begin{defn} \label{de:tg}
Let $M$ be a hypersurface in $\mathbb{S}^{n+1}$. Suppose that $\partial M$ is contained in a great sphere $\mathbb{S}^n\subset \mathbb{S}^{n+1}$. 
\begin{enumerate}
\item For  an open subset  $\Gamma\subset \partial M$, we say that \emph{$M$ is tangent to the great sphere $\mathbb{S}^n$ along  $\Gamma\subset \p M$} if $M$ is locally a graph of $u$  on $\mathbb{S}^n$ near $\Gamma$ such that $u=0, |Du|=0$ on $\Gamma$.  Furthermore, we say that $p\in \Gamma$ is a \emph{strictly convex boundary point} if there is a normal neighborhood $N$ of $p$ in $M$ such that $\exp_p^{-1}(N)$ is contained in a half space in $T_pM$, and  the only point in the closure of $\exp_p^{-1}(N)$ that intersects the boundary of the half-space is $\exp_p^{-1}(p)$.
\item  Suppose that $\partial M$ is contained in the hemisphere $\mathbb{S}_+^n$, then $M$ is said \emph{to be tangent to  $\mathbb{S}_+^n$ at $\partial M$ from the the region enclosed by $\p M$} if $\partial M$ encloses an open subset $V \subset \mathbb{S}_+^n$ so that $M$ is locally  the graph of a bounded function $u$ in a collar neighborhood of $\partial V$ in $V$ with $u=0, |Du|=0$ on $\partial V$. 
\end{enumerate}
\end{defn}

\begin{remark} \label{re:boundary}
If $p$ is a convex boundary point, then $M$ locally near $p$ lies in one side of a great $n$-sphere $\mathbb{S}^{n}$. Equivalently, there is a  stereographic projection  $\Phi$ which sends this great $n$-sphere onto the hyperplane $\{ x_1=0\}$ in $ \mathbb{R}^{n+1}$  such that  $\Phi(M)$ is locally a graph of $u$ in some open subset  $W\subset X_a$ such that $u=0$ and $|Du|=0$ on $\partial W$ and $\partial W \cap \{ (x_1, \dots, x_n) \in \mathbb{R}^n: x_1 = 0\} = \Phi(p)$, where $\partial W$ is the boundary of $W$ in $X_a$. 
\end{remark}
\begin{remark} \label{re:proj}
If $M$ satisfies $(2)$ in Definition~\ref{de:tg}, then there exists a stereographic projection $\Phi$  such that  $\mathbb{S}^n_+$ is mapped via $\Phi$ onto the closed half-space $\{ (x_1,\dots, x_n)\in \mathbb{R}^n: x_1 \ge 0\}$, and there is an open subset $U \subset  \{ (x_1,\dots, x_n)\in \mathbb{R}^n: x_1 \ge  0\}$ such that  $\Phi(\p M) = \partial U$, and  $\Phi(M)$ is locally a graph of  a bounded function $u$ in a collar neighborhood of $\partial U$ in $U$  such that $u = 0$ and $|Du| = 0$ on $\partial U$.
\end{remark}

\begin{corollary}\label{corollary:local}
Let $M$ be a connected, embedded, two-sided hypersurface in $\mathbb{S}^{n+1}$ with boundary $\partial M$. Suppose $\textup{int}(M)$ is $C^{2}$ and $M$ is $C^1$ to the boundary. Suppose that $M$ is tangent to the great sphere $\mathbb{S}^n$ along an open subset $\Gamma\subset \partial M$.  If $p\in \Gamma$ is a strictly convex boundary point and, in a neighborhood of $p$ in $M$,  the scalar curvature $R\ge n(n-1)$ and the mean curvature is weakly convex, then a neighborhood of $p$ in $M$ is contained  in $\mathbb{S}^n$.
\end{corollary}

\begin{proof}
Consider the isometric image $\Phi(M)$ in $(\mathbb{R}^{n+1}, g_S)$ via the stereographic projection $\Phi$  in Remark~\ref{re:boundary}.  For some $a>0$,  $M$ is locally the graph of $u$ in an open subset $W\subset  X_a$ so that the graph of $u$ has $R\ge n(n-1)$ and is weakly mean convex in $W$ with $u=0, |Du|=0$ on $\partial W$. By Proposition~\ref{pr:levelset}, $u\equiv 0$ in $W$.
\end{proof}

For a general hypersurface with $R\ge n(n-1)$ in $\mathbb{S}^{n+1}$ \emph{without} the a priori mean curvature assumption, we need to analyze the points where the mean curvature may change signs.  By Gauss equation, $R\ge n(n-1)$ implies $H^2 \ge |A|^2$, where $A$ is the shape operator of $M\subset \mathbb{S}^{n+1}$ and $H$ is the mean curvature. Hence the set of (interior) points with zero mean curvature is identical to the set of (interior) geodesic points $M_0 = \{ p \in \inte(M): A = 0 \mbox{ at } p\}$. 

The following lemma gives a useful characterization of $M_0$. The proof is analogous to the proof for hypersurfaces in Euclidean space due to R.~Sacksteder~\cite{Sacksteder:1960}  (cf. Lemma~\ref{le:geodesicE} below). 
\begin{lemma} \label{le:geodesicS}
Let $M$ be a $C^{n+1}$ hypersurface in $\mathbb{S}^{n+1}$ (possibly with boundary). Let $M_0'$ be a non-empty connected component of $M_0$. Then $M_0'$ is contained in a great $n$-sphere $\mathbb{S}^n$ of $\mathbb{S}^{n+1}$ so that $M$ is tangent to the sphere $\mathbb{S}^n$ at every point of $M_0'$.
\end{lemma}
\begin{pf}
Consider $\mathbb{S}^{n+1}$ as the unit sphere centered at the origin in Euclidean space $\mathbb{R}^{n+2}$. Define the (generalized) Gauss map $\nu: \inte(M) \rightarrow \mathbb{S}^{n+1}$ by assigning to $p \in \inte(M)$  the unit vector normal to $M$ in $T_p \mathbb{S}^{n+1} \subset T_p\mathbb{R}^{n+2}$.  
Since $M$ is of $C^{n+1}$, the map $\nu$ is of $C^n$. By direct computation, the rank of $\nu$ is zero at any point of $M_0$.  By Sard Theorem, the image $\nu(M_0)$ has $1$-dimensional Hausdorff measure zero in $\mathbb{S}^{n+1}$. It follows that $\nu(M_0)$ is totally disconnected in $\mathbb{S}^{n+1}$. Hence, $\nu(M_0')$ consists of a single point, denoted by $\nu_0$. Because the position vector $x(p)\in \mathbb{R}^{n+2}$ is orthogonal to $\nu_0$ at $p\in M_0'$, it implies that  $M_0'$  lies in the hyperplane that passes through the origin and is orthogonal to $\nu_0$. Hence $M_0'$ is contained in the intersection of the hyperplane and the sphere, which is a great $n$-sphere in $\mathbb{S}^{n+1}$ orthogonal to $\nu_0$.
\qed
\end{pf}

The above lemma says that the components where the mean curvature changes signs are flat. It allows us to employ a replacement argument, which is the key to remove the mean curvature assumption in Proposition~\ref{pr:levelset}.  

\begin{prop} \label{proposition:open}
Let $W$ be an open subset in $X_a$ for some $a>0$. Denoted by $\partial W$ the boundary of $W$ in $X_a$.  Let  $u \in C^{n+1}(W ) \cap C^1 (\overline{W})$  satisfy $u(x)=v \cdot x$ and $Du=v$ on $\p W$ for some $v\in \mathbb{R}^n$. If the graph of $u$ in $(\mathbb{R}^{n+1}, g_S)$ satisfies $R\ge n(n-1)$, then  $u(x)\equiv v\cdot x$ in $W$. 
\end{prop}
\begin{proof}
By Proposition~\ref{pr:levelset}, it suffices to prove that the mean curvature does not change signs. In fact, we shall show that $H(u)\equiv 0$ in $W$. Suppose on the contrary that $H(u)$ is not identically zero. We may without loss of generality assume that $\{ x\in W: H(u) > 0 \mbox{ at } x\}$ is non-empty, and let $\Omega$ be a connected component. Write $\{ (x_1,\dots, x_n)\in \mathbb{R}^n: x_1 < a\} \setminus \overline{\Omega} = U_0 \cup (\cup_{k>0} U_k)$ as the disjoint union of the connected components  where $U_0$ is the  component that contains the half-space $\{ (x_1,\dots, x_n)\in \mathbb{R}^n: x_1 < 0\}$. We show that $u$ can be replaced in each $U_k$ $(k> 0)$ by the graph of a great $n$-sphere.

By Proposition~\ref{proposition:connected},  each $\partial U_k$ is connected.  Because the graph of $u$ over $\partial U_k$ is contained in $M_0\cup \textup{Graph}(u|_{\partial W})$,  by Lemma~\ref{le:geodesicS} and the boundary condition of $u$, $\textup{Graph}(u|_{\partial U_k})$ is contained in a great $n$-sphere $\mathbb S^n$ such that the graph of $u$ is tangent to the great $n$-sphere. Hence there is a unique $h_k\in C^{\infty}(U_{k})$ for each $k>0$ where $h_k$ is a graph function of some great $n$-sphere such that  $u = h_k, Du = Dh_k$ on $\partial U_k$ (see Appendix~\ref{section:sphere}). Furthermore $D^2 u = D^2 h_k $ on $\partial U_k \cap W$ since the shape operators of the graphs are both zero over $\partial U_k \cap W$. 

Define the function $\tilde{u}$ in $W_0:=\{ (x_1,\dots, x_n)\in \mathbb{R}^n: x_1 < a\} \setminus \overline{U_0}$ such that $\tilde{u} = u$ in $\overline{\Omega}$ and $\tilde{u} = h_k$ in $U_k$ for each $k>0$. Then $\tilde{u}$ is $C^2(W_0) \cap C^1(\overline{W_0})$, and the graph of $\tilde{u}$ satisfies $R\ge n(n-1)$ and is weakly mean convex everywhere, with strict mean convex $H(\tilde{u})>0$ in $\Omega$. 

We show that such $\tilde{u}$ cannot exist.  If $\partial W_0 $ intersects $\partial W$ or if $\partial W_0$ is unbounded, then $\tilde{u}(x) =\tilde{v} \cdot x$ and $D\tilde{u} =\tilde{v}$  on $\partial W_0$ on $\partial W_0$ for some $\tilde{v} \in \mathbb{R}^n$, and hence by Proposition~\ref{pr:levelset}, the graph of $\tilde{u}$ must have zero mean curvature, but it contradicts that $H(\tilde{u})>0$ in $\Omega$.  If $\partial W_0$ does not intersect $\partial W$ and is bounded, there there is $y =(y_1,\dots, y_n)\in \partial W_0$ with $y_1>0$ such that $W_0\subset \{ (x_1, \dots, x_n) \in \mathbb{R}^n: x_1 \ge y_1>0\}$. By Lemma~\ref{lemma:boundary}, $(y, \tilde{u}(y))$ is a strictly convex boundary point, and hence by Corollary~\ref{corollary:local}, the graph of $\tilde{u}$ is contained in the great $n$-sphere, which contradicts that $H(\tilde{u})>0$ in $\Omega$.
\end{proof}

\begin{proof}[Proof of Theorem~\ref{th:rigidityS}] 
Using a stereographic projection $\Phi$ in Remark~\ref{re:proj}, there is an open subset $U=\{ (x_1, \dots, x_n): x_1>0\}$ such that $\Phi(\partial M) = \partial U$ and $\Phi(M)$ is locally the graph of $u$ in a collar neighborhood of $U$ such that $u=0, |Du|=0$ on $\partial U$. Let $I = \{ (0, a) \subset \mathbb{R}^+: u\equiv 0 \mbox{ on } U\cap X_a\}$. The interval $I$ is closed by continuity of $u$. By Proposition~\ref{proposition:open}, $I$ is non-empty and open. This implies that $I= \mathbb{R}^+$ and hence $u\equiv 0$ on $U$; that is, $M$ is contained in  the great $n$-hemisphere. 
\end{proof}

\section{Rigidity in Euclidean space} \label{se:rigidityE}
The arguments in the previous section can  be applied to hypersurfaces with boundary in Euclidean space, which extends our earlier work for \emph{complete} hypersurfaces in Euclidean space \cite{Huang-Wu:2013}.  We also refine some results in \cite{Huang-Wu:2013}. In this section, $H(u)$ is the mean curvature of the graph of $u$ in Euclidean space with respect to the upward unit vector $\nu = \frac{(-\nabla u, 1)}{\sqrt{1+|\nabla u|^2}}$:
\[	
	H(u)=\sum_{i,j=1}^n \left( \delta_{ij} - \frac{u_i u_j}{ 1+ |Du|^2} \right) \frac{u_{ij} }{ \sqrt{1+|Du|^2}}.
\] 

Recall $X_a = \{ (x_1, \dots, x_n) \in \mathbb{R}^n:  0\le x_1 < a\}$. By the strong maximum principle for the mean curvature operator in the same way as in Lemma~\ref{lemma:pos}, we have the following lemma. 

\begin{lemma}[Cf. {\cite[Proposition 3.1]{Huang-Wu:2013}}] \label{lemma:posE}
Let $W$ be an open subset in $X_a$ for some $a>0$, and let $\partial W$ be the boundary of $W$ in $X_a$.  Let $u \in C^2( W) \cap C^1 (\overline{W})$ satisfy that $u(x)=v\cdot x +b$ and $Du=v$ on $ \p W\cap B(p)$ for some $v \in \mathbb{R}^n$, $b\in \mathbb{R}$.  If $H(u) \ge 0 $ in $ W$, then $u(x) > v\cdot x+b$ somewhere in $W$, unless  $u\equiv v\cdot x+b$ in $W$. 
\end{lemma}

\begin{proposition}[Cf. {\cite[Lemma 3.5]{Huang-Wu:2013}}] \label{proposition:mean-convex}
Let $W$ be an open subset in $X_a$, and let $\partial W$ be the boundary of $W$ in $X_a$.  Let  $u \in C^2(W ) \cap C^1 (\overline{W})$ be bounded and  satisfy $u(x)=v \cdot x + b$ and  $Du=v$ on $\p W$ for some $v \in \mathbb{R}^n, b\in \mathbb{R}$. Suppose the graph of $u$ in the Euclidean space $\mathbb{R}^{n+1}$ satisfies $R\ge 0$ and either $H(u)\ge 0$ or $H(u)\le 0$  in $W$. Then one of the following holds:
\begin{enumerate}
\item $u(x) =v \cdot x + b$ in  $W $; that is, the graph of $u$ is contained in a hyperplane.
\item  The graph of $u$ is a generalized cylinder. 
\end{enumerate}
In particular, if $\partial W$ is bounded, then $(1)$ must hold because a nontrivial generalized cylinder cannot  satisfy the required boundary condition. 
\end{proposition}
\begin{proof}
Because rotation and translation are  rigid motions in Euclidean space, we may assume $v=0$ and $b=0$. We also assume $H(u)\ge 0$, for otherwise replace $u$ by $-u$.  

For an arbitrary fixed $ a_0 \in (0, a)$ such that $W\cap X_{a_0}$ is not empty, if $u$ is not identically zero in $W\cap X_{a_0}$, then by Lemma~\ref{lemma:posE}, $u>0$ at some point in $W\cap X_{a_0}$. Let $\psi_\lambda(x) = \lambda x_1$ for some $\lambda \in \mathbb{R}^n$. Because of the boundary condition of $u$ and that $u$ is bounded, $\psi_\lambda > u $ for $\lambda$ sufficiently large. We continuously decrease $\lambda$ until that the graphs of $\psi_\lambda$ and $u$ touch for the first time at $p\in W\cap X_{a_0}$. By Lemma~\ref{lemma:posE}, we have $u(p) >0$ and hence $\lambda >0$. Also since the graph of $\psi_\lambda(x)$ has zero mean curvature, it cannot be tangent to the graph of $u$ at an interior point by maximum principle, and thus $p\in \{ (x_1, \dots, x_n) \in \mathbb{R}^n: x_1=a_0\}$. 

Since $p$ is the first contact point, 
\[
	|Du|(p) \ge \partial_{x_1} u (p) \ge \partial_{x_1} \psi_\lambda(p)  = \lambda>0. 
\]
Let $\Sigma= u^{-1}(u(p))$ the level set in $W$ and $\widetilde{\Sigma} = \psi_\lambda^{-1} (u(p)) $. Note that $\Sigma$ is $C^2$ near $p$ since $|Du|(p)>0$.  Let $H_\Sigma$ be the mean curvature scalar of $\Sigma$ of the unit normal vector $\eta:= (\nabla u, 0) / |\nabla u|$. By comparison principle, since $\widetilde{\Sigma}$ is tangent to $\Sigma$ at $p$ toward the common normal vector $\partial_{x_1}$ at $p$ and the mean curvature of $\widetilde{\Sigma}$ is zero,  we have $H_\Sigma \ge 0$ at $p$. 

On the other hand, by Corollary~\ref{co:eqE} (with $R\ge 0, H(u)\ge 0$, $\langle \nu, \eta \rangle <0$), we obtain $H_\Sigma\le 0$. We then conclude that $H_\Sigma=0$ and furthermore $\Sigma$ is identical to $\tilde{\Sigma}$ by strong maximum principle, which also implies that $W \supset \{ (x_1, \dots, x_n)\in \mathbb{R}^n:x_1=a_0\}$. By varying $a_0$, it implies that the graph of $u$ depends only on $x_1$ and hence is a generalized cylinder. 
\end{proof}

We follow the replacement argument as in Proposition~\ref{proposition:open} and remove the mean convexity assumption. In particular, our result includes generalized cylinders that may not be weakly mean convex. Let $M$ be a hypersurface in Euclidean space with nonnegative scalar curvature. By the Gauss equation, the mean curvature $H$ is zero only at the geodesic points. Recall the following result  
\begin{lemma}[\cite{Sacksteder:1960},  see also {\cite[Lemma 3.6]{Huang-Wu:2013}}]\label{le:geodesicE} 
Suppose  $M$ is a $C^{n+1}$ hypersurface in $\mathbb{R}^{n+1}$ (possibly with nonempty boundary). Denote by $M_0 = \{ p \in \textup{int}(M): A = 0 \mbox{ at } p\}$ the set of interior geodesic points. Let $M_0'$ be a connected component of $M_0$. Then $M_0'$ lies in a hyperplane which is tangent to $M$ at every point in $M_0'$. 
\end{lemma}

\begin{proposition}[Cf. {\cite[Proposition 3.8]{Huang-Wu:2013}}] \label{proposition:rigidity}
Let $W$ be an open subset in $X_a$, and let $\partial W$ be the boundary of $W$ in $X_a$.  Let  $u \in C^{n+1}(W ) \cap C^1 (\overline{W})$  be bounded and satisfy $u(x)=v \cdot x + b$ and  $Du=v$ on $\p W$ for some $v \in \mathbb{R}^n$ and $b\in \mathbb{R}$. Suppose the graph of $u$ in the Euclidean space $\mathbb{R}^{n+1}$ satisfies $R\ge 0$. Then either one of the following holds:
\begin{enumerate}
\item $u(x) =v \cdot x + b$ in  $W $; that is, the graph of $u$ is contained in a hyperplane.
\item  The graph of $u$ is a generalized cylinder. 
\end{enumerate}
In particular, if $\partial W$ is bounded, then $(1)$ must hold because a nontrivial generalized cylinder cannot  satisfy the boundary assumption. 
\end{proposition}
\begin{proof}
By Euclidean rigid motion, we may assume $v=0$ and $b=0$. If $H(u) \ge 0$ or $H(u)\le 0$, then Proposition~\ref{proposition:mean-convex} applies. If $H(u)$ changes signs, let $\Omega$ be a connected component of $\{ x\in W: H(u) \neq 0 \mbox{ at } x\}$.  Write $\{ (x_1, \dots, x_n)\in \mathbb{R}^n: x_1 < a\} \setminus \overline{\Omega} = U_0 \cup (\cup_{k>0} U_k)$ as the disjoint union of connected components, where $U_0$ is the component that contains $\{ (x_1, \dots, x_n) \in \mathbb{R}^n: x_1 < 0\}$. By Proposition~\ref{proposition:connected}, $\partial U_k$ is connected. Note either $\partial U_k$ intersects with $\partial W$ or $\partial U_k$ contains only points in $W$. By Lemma~\ref{le:geodesicE} and the boundary condition of $u$,  the graph of $u$ on $\partial U_k$  lies in a hyperplane, say the graph of a linear function $h_k$, so  that the graph of $u$ is tangent to the hyperplane.

We define $\tilde{u}$ in $W_0:=\{ (x_1, \dots, x_n)\in \mathbb{R}^n: x_1 < a\} \setminus \overline{U_0}$ by $\tilde{u} = u$ in $\overline{\Omega}$ and $\tilde{u} = h_k$ in $U_k$. Clearly $\tilde{u} \in C^1(\overline{W_0})$. To see that  $\tilde{u} \in C^2(W_0)$, we note that $D^2 u = 0 $ on $\partial \overline{\Omega}\cap W$ because the shape operator of the graph of $u$ is zero there. Therefore, $\tilde{u}$ satisfies the assumptions in Proposition~\ref{proposition:mean-convex}. This implies that $\Omega = \{ (x_1, \dots, x_n) \in \mathbb{R}^n: c_1< x_1 < c_2\}$ for some constants $c_1, c_2$ and  $\tilde{u} = u$ is a function of $ x_1$ in $\Omega$.  We repeating the argument for other connected components. For the set of points with zero mean curvature, we apply Lemma~\ref{le:geodesicE}. It implies that $u$ is a function that depends only on $x_1$ and hence $(2)$ holds. 
\end{proof}

\begin{thm-rigidityE}
Let $M$ be a connected, embedded, two-sided hypersurface  in $\mathbb{R}^{n+1}$ with $R\ge 0$ and with non-empty boundary $\partial M$. Suppose $\textup{int}(M)$ is $C^{n+1}$ and $M$ is $C^1$ up to boundary. Suppose that $\p M $ is contained in the hyperplane $\{ x_{n+1}=0\}$ and $ \partial M= \partial U$ for some open subset  $U \subset \{(x_1,\dots, x_n)\in \mathbb{R}^n:  x_1 > 0 \}$ such that $M$ is locally the graph of a bounded function $u$ in a collar neighborhood of $\partial U$ in $U$  with $u = 0$ and $|Du|=0$ on $\p U$. Then $M$ is a portion of either the hyperplane $\{ x_{n+1}=0\}$ or a generalized cylinder. In particular, if $\partial M$ is bounded, then $M$ must be a portion of the hyperplane. 
\end{thm-rigidityE}
\begin{proof}
Consider the interval $I=(0, a] \subset \mathbb{R}^+$ such that either $u=0$ in $U\cap X_a$ or the graph of $u$ is a generalized cylinder on $ U\cap X_a$. It is closed by continuity of $u$. By Proposition~\ref{proposition:rigidity}, $I$ is  non-empty and open and hence $I= \mathbb{R}^+$.
\end{proof}

\section{Examples} \label{se:Examples}
In this section, we shall present two examples to demonstrate that the boundary conditions in Theorem~\ref{th:rigidityS} and Theorem~\ref{th:rigidityE} are necessary. 

\begin{figure}[here] 
   \centering
   \includegraphics[width=0.7\textwidth]{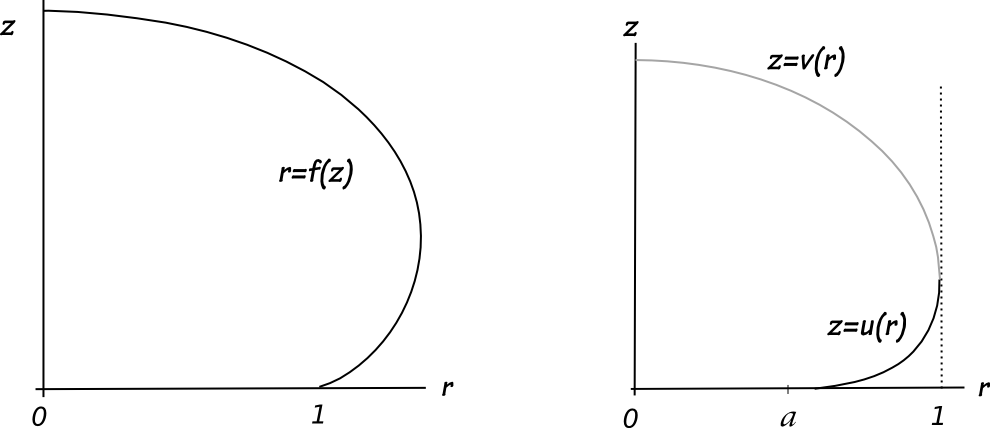}
   \caption{The left figure indicates the graph of the function $r=f(z)$ in $\mathbb{R}^3$ in Example \ref{ex:E}  where $r = \sqrt{x^2 + y^2}$. The right figure indicates the graphs of $z=u(r)$ and $z = v(r)$ in $\mathbb{S}^3$ in Example \ref{ex:S}. The surfaces are obtained by rotating the curves about the $z$-axis.}
\end{figure}

Let us first consider the boundary conditions in Theorem~\ref{th:rigidityE}.  In particular, if $\partial M$ is bounded, the assumption says that $M$ is tangent to the hyperplane at $\p M$ from the region \emph{enclosed} by $\p M$. This condition is necessary, as shown in the following example that $M$ has nonnegative scalar curvature but  is tangent to $\p M$ from the region \emph{outside} of $\p M$.
\begin{ex}\label{ex:E}
	Consider the surface in Euclidean space $(\mathbb{R}^3, g_0)$:
\[
   M = \{(x,y,z)\in \mathbb{R}^3\mid \sqrt{x^2 + y^2} - f(z) = 0 \},
\]
where
\[
   f(z) = (\sqrt{z} + 1)\sqrt{1 - z^2} \qquad \textup{for all $0 \le z \le 1$}.
\]
Clearly, the surface $M$ is obtained by rotating the curve $f(z)$ about the $z$-axis. Note that $M$ has nonnegative Gauss curvature, because $f(z)$ is concave. Furthermore, $M$ is smooth with boundary
\[
   \p M = \{(x,y,0) \in \mathbb{R}^3 \mid x^2 + y^2 = 1\}.
\]
\qed
\end{ex}

Next, we consider the assumption in Theorem~\ref{th:rigidityS} that  $M$ is tangent to a great $\mathbb{S}^n_+$ at $\p M$ from the region \emph{enclosed} by $\partial M$ (see Definition~\ref{de:tg} and Remark~\ref{re:proj}). The following example shows that the assumption cannot be removed.
\begin{ex} \label{ex:S}
  Consider $\mathbb{R}^3$ with the spherical metric $g_S$, i.e., 
  \[
     g_S = \frac{4g_0}{(1 + x^2 + y^2 + z^2)^2},
  \]
where $g_0$ is the Euclidean metric and $(x,y,z)$ are the Cartesian coordinates. We denote $r = \sqrt{x^2 + y^2}$. Fix $ a \in (0, 1)$. For $r \in (a,1)$,
\[
   u(r) = \frac{1}{\sqrt{2}} \left( - 2\sqrt{1 - r}  - \frac{r}{\sqrt{1 - a}} + 2 \sqrt{1 - a}+ \frac{a}{\sqrt{1 -a}} \right).
\]
For $r\in [0,1]$, 
\[
   v (r) = u(1) + \sqrt{1 - r^2} = \sqrt{\frac{1-a}{2}} + \sqrt{1 - r^2}.
\]
Let $M$ be the surface which is the union of graph of $u$ over $[a,1]$ and the graph of $v$ over $[0,1]$. We \emph{claim} that $M$ is a $C^2$ surface of scalar curvature $R \ge 2$ in $(\mathbb{R}^3,g_S)$ and that $R=2$ holds at and only at the boundary points. 

Note that the graph of $v$ is a portion of the unit $2$-sphere centered at $(r,z) = (0,\sqrt{(1-a)/2})$. With respect to the upward unit normal, the graph of $v$ has principal curvatures
\[
   \kappa_1 = \kappa_2 = - \frac{[u(1)]^2}{2} = -\frac{1-a}{4}.
\] 
Hence, $R = 2 + 2\kappa_1\kappa_2 > 2$ on the graph of $v$. For the function $u$, obviously we have $u \in C^{\infty}([a,1))\cap C^{0}([a,1])$. In addition, $u$ satisfies the following properties:
\begin{align}
   u(a) & = 0, \quad u'(a)  = 0, \quad u'(r) > 0 \quad \textup{for all $a < r < 1$, \; and}  \label{eq:du}\\
   u''(r) & > u' [1 + (u')^2], \qquad \textup{for all $a \le r < 1$}.    \label{eq:d2u}
\end{align}
With respect to the upward unit normal, the principle curvatures of the graph of $u$ are given by
\begin{align*}
	\lambda_1 &= \frac{1}{\sqrt{1+(u')^2}} \left[ u - ru'+\frac{1+u^2+r^2}{2}\frac{u''}{1+(u')^2}\right]\\
	\lambda_2 &=\frac{1}{\sqrt{1+(u')^2}} \left[ u - ru'+\frac{1+u^2+r^2}{2}\frac{u'}{r}\right] \\
	&=\frac{1}{\sqrt{1+(u')^2}} \left[ u + u'\frac{1+u^2-r^2}{2}\right].
\end{align*}
Since $r \le 1$, by \eqref{eq:du} we have that $\lambda_2 \ge 0$ and  $\lambda_2 = 0$ if and only if $r = a$. Applying \eqref{eq:d2u} to  $\lambda_1$ yields that
\begin{align*}
	 &u - ru'+\frac{1+u^2+r^2}{2}\frac{u''}{1+(u')^2} \\
	 &> u - ru' + \frac{1+u^2 + r^2}{2}u' 
	 = u+ u'\frac{u^2 + (r-1)^2}{2}\ge 0.
\end{align*}
It follows that $\lambda_1 > 0$. Therefore, $R = 2 + 2\lambda_1 \lambda_2 \ge 2$ on the graph of $u$, where $R = 2$ if and only if $r = a$. Obviously $M$ is smooth at the interior points of graph $u$ and graph $v$. It is elementary to verify that $M$ is of $C^2$ at the intersection curve $(r,z) = (1,\sqrt{(1-a)/2})$ of the two graphs. Thus, the claim is proved. 

Finally, we can perturb $M$ to  obtain a smooth surface. Because $R>2$ in a neighborhood of the intersection curve, we can perturb $M$ locally near the intersection curve to get a smooth surface $\widetilde{M}$ in $(\mathbb{R}^3,g_S)$ with $R\ge 2$. Also, $\widetilde{M}$ is identical to $M$ away from a neighborhood of the intersection curve. Hence, the smooth surface $\widetilde{M}$ satisfies $R\ge 2$ but the boundary is tangent from the \emph{exterior}.
\qed
\end{ex}

\appendix
\section{Topological properties} \label{se:appA}

We denote by $\tilde{H}_k(X)$ the $k$th reduced homology group of a topological space $X$ with coefficients in $\mathbb{Z}$. Note that $\tilde{H}_0(X)$ is a free abelian group, and  the rank of $\tilde{H}_0(X)$ plus $1$ is the number of path-connected components of $X$. Hence, $\tilde{H}_0(X)=0$ if and only if $X$ is path-connected. We recall the following result in \cite{Huang-Wu:2013}, which follows from the Mayer-Vietoris sequence.
\begin{lemma}[{\cite[Lemma A.1]{Huang-Wu:2013}}] \label{le:Euclidean}
Let $X$ be a contractible topological space. Let $U, V$ be two subsets of $X$ so that $X = \textup{int}(U) \cup \textup{int}(V)$ (the interiors of $U$ and $V$ may intersect). Then 
\[
	\tilde{H}_0 (U\cap V) \approx \tilde{H}_0 (U) \oplus \tilde{H}_0(V),
\]
where $\approx$ stands for the group isomorphism. In particular, if $U$ and $V$ are path-connected, then $U\cap V$ is also path-connected.
\end{lemma}

\begin{proposition}[{\cite[Proposition A.3]{Huang-Wu:2013}}] \label{proposition:connected}
Let $X$ be a contractible locally-connected topological space. Let $\Omega \subset X$ be a non-empty  connected open subset. Let $X\setminus \overline{\Omega} = \cup U_k$ be the  disjoint union of connected components.  Then $\partial U_k$ is connected. 
\end{proposition}

We include our proof here and also fix some minor typos in \cite[Proposition A.3]{Huang-Wu:2013}. 
\begin{proof}
Note that each $U_k$ is open and $\partial U_k \subset \partial \overline{\Omega}$. By replacing $\overline{\Omega}$ with $\overline{\Omega} \cup (\cup_{k \neq k_0} U_{k})$ (which is also closed and connected), we may assume that $X\setminus \overline{\Omega}= U$ is connected.  Let $\Sigma = \partial\overline{\Omega} = \partial U$.  Suppose on the contrary that $\Sigma$ is not connected. There exist non-empty disjoint open subsets $E, F$ in $X$ so that both $E$ and $F$ intersect $\Sigma$ and $\Sigma \subset E\cup F$. By discarding some components of $E$ and $F$ if necessary, we assume that every component of $E$ and $F$ intersect $\Sigma$, and note they also intersect $\overline{\Omega}$ and hence $\Omega$. Therefore, $\Omega\cup (E\cup F)$ is connected and hence path-connected since it is an open subset (and hence locally-connected). Similarly, $U\cup (E\cup F)$ is also path-connected.  However, it contradicts Lemma~\ref{le:Euclidean}, which implies that $E\cup F$ is path-connected:
\[
		\tilde{H}_0 (E\cup F) \cong \tilde{H}_0 (\Omega\cup (E\cup F)) \oplus \tilde{H}_0(U\cup (E\cup F)).
\]
 
\end{proof}

\section{Spherical geometry}\label{section:sphere}
We include some facts on the geometry of $\mathbb{R}^{n+1}$ endowed with the spherical metric $g_S= \phi^{-2} g_0$ where
\[
	\phi = \frac{(1 + \sum_{i=1}^{n+1} (x_i)^2 )}{2}.
\]
By the stereographic projection, $(\mathbb{R}^{n+1}, g_S)$  is isometric to $\mathbb{S}^{n+1}$ with a point removed.  Let $(x_1, \cdots, x_{n+1})$ be the Cartesian coordinates in $\mathbb{R}^{n+1}$ and $|x|^2 = \sum_{i=1}^{n+1} x_i^2$. A hypersurface $S$ in $(\mathbb{R}^{n+1}, g_S)$ is isometric to a great $n$-sphere in $\mathbb{S}^{n+1}$ if and only if either one of the following holds:
\begin{enumerate}
\item $S$ is a hyperplane through the origin; that is, the graph of $h(x_1,\dots, x_n)= b_1 x_1 +\cdots+ b_n x_n$ for some $(b_1,\dots, b_n) \in \mathbb{R}^n$.
\item $S$ is  a sphere centered at $(a_1, \dots, a_{n+1})\in \mathbb{R}^{n+1}$ of radius $\sqrt{1+|a|^2}$; that is, the union of the graphs $h_{\pm}$ where $|a| = a_1^2+\dots + a_{n+1}^2$ and
\[
	h_{\pm}(x_1, \dots, x_n)= a_{n+1} \pm \sqrt{1+|a|^2 - \sum_{i=1}^{n}( x_i - a_i)^2}.
\]
\end{enumerate}

\begin{lemma}\label{lemma:boundary}
Let $p= (p_1, \dots, p_{n+1}) \in \mathbb{R}^{n+1}$ with $p_1>0$. Then  there exists a  unique great $n$-sphere $S$ such that $S$ passes through $p$ and $S\setminus \{ p \} \subset \{ (x_1,\dots, x_{n+1}): x_1 <p_1 \}$. As a consequence, if $p$ is the boundary point of a hypersurface that lies in the half space $\{ (x_1,\dots, x_{n+1}): x_1 \ge p_1 \}$, then $p$ is a strictly convex boundary point. 
\end{lemma}
\begin{proof}
The sphere centered at $(a_1, p_2, \dots, p_{n+1})$ of radius $\sqrt{1+a_1^2 + p_2^2+ \dots + p_{n+1}^2}$ with 
\[
	a_1 = \frac{p_1}{2} - \frac{1}{2p_1} ( 1+ p_2^2 +\dots + p_{n+1}^2)
\]
satisfies the desired properties. 
\end{proof}

\bibliographystyle{amsplain}
\bibliography{2016}

\end{document}